\documentclass[graybox,draft]{svmult}

\usepackage{amsfonts,amstext,amssymb,physics} 
\usepackage{amsmath}

\usepackage{amssymb}
\usepackage{mathptmx}       
\usepackage{helvet}         
\usepackage{courier}        
\usepackage{type1cm}        
\usepackage{graphicx}        
\usepackage{multicol}        
\usepackage[bottom]{footmisc}

\usepackage{enumerate}

\usepackage{enumitem}  
\usepackage{calc}
\setlist{labelindent=1pt,itemsep=0.1cm}
\setlist[itemize]{leftmargin=0.7cm}
\setlist[enumerate]{itemindent=0em,leftmargin=0.7cm}

\usepackage{float} 

\usepackage[bookmarks=false,hidelinks,unicode]{hyperref}

\smartqed

\allowdisplaybreaks

\makeindex             

\begin{document}
\title*{Fixed point results for set-contractions on dislocated metric space with a directed graph}
\titlerunning{Fixed point results for set-contractions on dislocated metric space} 
\author{Talat Nazir, Zakaria Ali, Shahin Nosrat Jogan \and Sergei Silvestrov}
\authorrunning{T. Nazir, Z. Ali, S. N. Jogan, S. Silvestrov} 

\institute{Talat Nazir, Zakaria Ali, Shahin Nosrat Jogan
\at
Department of Mathematical Sciences, University of South Africa, Florida 0003, South Africa. \\ \email{talatn@unisa.ac.za, alizi@unisa.ac.za, nosgo.bots@gmail.com}
\and
Sergei Silvestrov
\at Division of Mathematics and Physics, School of Education, Culture and Communication, M{\"a}lardalen University, Box 883, 72123 V{\"a}ster{\aa}s, Sweden. \\ 
\email{sergei.silvestrov@mdu.se}
}
%
%


\maketitle
\label{chap:NazirAliJoganSilvestrov:FPRDMS}

\abstract*{Generalized (rational) graph contractions in the framework of a dislocated metric space endowed with a directed graph are investigated. Fixed point results for set-contractions are obtained. We also provide some examples to illustrate our main results. Moreover, the well-posedness of obtained fixed point results are also shown. Our obtained results extend many results in the existing literature.
}

\abstract{Generalized (rational) graph contractions in the framework of a dislocated metric space endowed with a directed graph are investigated. Fixed point results for set-contractions are obtained. We also provide some examples to illustrate our main results. Moreover, the well-posedness of obtained fixed point results are also shown. Our obtained results extend many results in the existing literature.
}

\section{Introduction to Dislocated Metric Space}
\label{IntrodDislocatedMetricNaS1:1.2.1}

Hitzler and Seda \cite{hitzler2000dislocated} introduced the notion of dislocated
metric space and proved a fixed point result in the given space. After this, Rasham et al. \cite{rasham2018multivalued} obtained the multivalued fixed point results for new generalized-dominated contractive mappings on dislocated metric space with application. Wadkare et al. \cite{wadkar2017some} proved some fixed point theorems in dislocated metric space. After which several useful
results were established,  \cite{abbas2019nadler, aydi2018nadler, karapinar2013dislocated, kumari2015completion, pasicki2015dislocated} introducing the notion of contractions and related fixed point results in dislocated metric spaces.

\begin{definition}[Dislocated metric space, \cite{hitzler2000dislocated}]
	\label{DefNaS1:1.2.1}
	Let $Y$ be a non-empty set and let $\xi: X \times X \rightarrow \mathbb{R_{+}}$ be a real-valued non-negative function, satisfying for all $r, s, t \in Y$,
	\begin{enumerate}[label=\textup{(\roman*)}, ref=(\roman*)]
		\item $\xi(r, s) = 0 \Longrightarrow r = s$; \label{dislocated_ax-d1}
		\item $\xi(r, s) = \xi(s, r)$; \label{dislocated_ax-d2}
		\item $\xi(r, t) \leq \xi(r, s) + \xi(s, t)$. \label{dislocated_ax-d3}
	\end{enumerate}
	Then the pair $(Y, \xi)$ is called a dislocated metric space.
\end{definition}
\begin{remark}
	\label{RmkNaS1:1.2.1}
	Note that the converse of \ref{dislocated_ax-d1} is only true in the usual metric space. Furthermore, every (usual) metric space is a dislocated metric space.
\end{remark}
One of the simplest examples of a dislocated metric space is $\xi(r, s) = \max\{r, s\}$ for all $r, s \in Y$. We build another dislocated metric using the previous example with the discrete metric. As follows:
\begin{example}
	\label{ExmNaS1:1.2.1}
	Let $Y = \mathbb{R_{+}}$, we define $\xi(r, s) = \max\{r, s\} + \sigma(r, s)$, where
	\begin{align*}
		\sigma(r, s) = \begin{cases}
			1 \text{ if } r \neq s, \\
			0 \text{ if } r = s.
		\end{cases}
	\end{align*}
	Note that $\xi(0, 0) = 0$, and $r = s \neq 0 \Longrightarrow \xi(r, s) > 0$. Proving the other two axioms is routine.
\end{example}

\begin{definition}
	\label{DefNaS1:1.2.2}
Let $(Y, \xi)$ be a dislocated metric space and $\epsilon \in \mathbb{R^{+}}$. We define the open ball as
	\begin{align*}
		B_{\epsilon}(s) = \{t \in X : |\rho(s, t) - \rho(s, s)| < \epsilon \}.
	\end{align*}
\end{definition}
The topology $\tau_{\xi}$ on $(Y, \xi)$ is defined then by open sets as follows:
\begin{align*}
	\tau_{\xi} = \{U \subseteq Y : \forall u \in U \exists \epsilon \in \mathbb{R_{+}} \text{ such that } B_{\epsilon}(u) \subseteq U \}.
\end{align*}
\noindent A sequence $\{y_{n}\} \subseteq (Y, \xi)$ is said to \emph{converge} to $y \in Y \iff \lim\limits_{n \rightarrow \infty} \xi(y_{n}, y) = \xi(y, y)$.
\begin{definition}
	\label{DefNaS1:1.2.3}
	Let $(Y, \xi)$ be a dislocated metric space. Then a sequence $(y_{n}) \subseteq Y$ is said to be a Cauchy sequence, if $\lim\limits_{m, n \rightarrow \infty} \xi(y_{n}, y_{m})$ exists and is finite. Furthermore, $(Y, \xi)$ is said to be complete if and only if every Cauchy sequence $(y_{n})$ in $Y$ converges to a point $y \in Y$ such that $\lim\limits_{n \rightarrow \infty} \xi(y_{n}, y) = \lim\limits_{m, n \rightarrow \infty} \xi(y_{n}, y_{m}) = \xi(y, y)$.
\end{definition}

\begin{definition}
	\label{DefNaS1:1.2.4}
	Let $(Y, \xi)$ be a dislocated metric space. A subset $U \subseteq Y$ is said to be bounded if there exists a real number $\mathfrak{L} \geq 0$ and $u_{0} \in Y$ such that $U \subseteq B_{\mathfrak{L}}(u_{0})$, that is, $|\xi(u_{0}, u) - \rho(u_{0}, u_{0})| < \mathfrak{L}$ for all $u \in U$.
\end{definition}

\begin{definition}
	\label{DefNaS1:1.2.5}
	Let $(Y, \xi)$ be a dislocated metric space. A subset $U$ of $Y$ is said to be closed in $Y$, if $\overline{U} = U$. That is $u \in \overline{U}$ if and only if there exists a sequence $(u_{n})$ in $U$ such that $\lim\limits_{n \rightarrow \infty} \xi(u_{n}, u) = \xi(u, u)$.
\end{definition}
Following the above definitions, for a dislocated metric space $(Y, \xi)$, let $CB^{\xi}(Y)$ be the set of all closed and bounded subsets of the set $Y$.  We again provide the definition for the Pompeiu-Hausdorff induced by the dislocated metric $\xi$.
\begin{definition}
	\label{DefNaS1:1.2.6}
	Let $(Y, \xi)$ be a dislocated metric space and $CB^{\xi}(Y) \neq \emptyset$. For any $U, V \in CB^{\xi}(Y)$, the Pompeiu-Hausdorff $PH$-metric $\mathcal{H}_{\xi}: CB^{\xi}(Y) \times CB^{\xi}(Y) \rightarrow \mathbb{R_{+}}$ induced by $\xi$ is defined as
	\begin{align*}
		\mathcal{H}_{\xi}(U, V) = \max \biggl\{\mathcal{D}_{\xi}(U, V), \mathcal{D}_{\xi}(V, U) \biggr\},
	\end{align*}
	where $\mathcal{D}_{\xi}(A, B) = \sup\{\xi(a, B) : a \in A \}$ and $\xi(a, B) = \inf\{\xi(a, b) : b \in B \}$.
\end{definition}
We will be using the following results which also state the properties of the $PH$-metric, as proven by Aydi et al. in \cite{aydi2018nadler}, in the proof of our main result.
\begin{lemma}[\cite{aydi2018nadler}]
	\label{LemNaS1:1.2.1}
	Let $(Y, \xi)$ be a dislocated metric space, $\emptyset \neq U \subseteq Y$, and $(y_{n})$ be a sequence in $Y$ and $y \in Y$. Then,
	\begin{enumerate}[label=\textup{\arabic*)}, ref=\arabic*]
		\item $\xi(u, U) = 0 \Longrightarrow u \in \overline{U}$;
		\item $\lim\limits_{n \rightarrow \infty} \xi(y_{n}, y) = \rho(y, y) \Longrightarrow \lim\limits_{n \rightarrow \infty} |\xi(y_{n}, U) - \xi(y, U)| = \rho(y, y)$.
	\end{enumerate}
\end{lemma}
\begin{theorem}[\cite{aydi2018nadler}]
	\label{ThmNaS1:1.2.1}
	Let $(Y, \xi)$ be a dislocated metric space. For $U, V, W \in CB^{\xi}(Y)$,
	\begin{enumerate}[label=\textup{\arabic*)}, ref=\arabic*]
		\item $\mathcal{H}_{\xi}(U, U) = \mathcal{D}_{\xi}(U, U) = \sup\{\xi(u, U) : u \in U\}$;
		\item $\mathcal{H}_{\xi}(U, V) = \mathcal{H}_{\xi}(V, U)$;
		\item $\mathcal{H}_{\xi}(U, V) = 0 \Longrightarrow U = V$;
		\item $\mathcal{H}_{\xi}(U, W) \leq \mathcal{H}_{\xi}(U, V) + \mathcal{H}_{\xi}(V, W)$.
	\end{enumerate}
\end{theorem}

We prove the following result.

\begin{lemma}
	\label{CorNaS1:1.2.1}
	Let $(Y, \xi)$ be a dislocated metric space. Then, $\mathcal{H}_{\xi}(U, U) \leq \mathcal{H}_{\xi}(U, V)$ for all $U, V \in CB^{\xi}(Y)$.
\end{lemma}

\begin{proof}
	Since, for all $ u \in U$, it follows that
	$ \xi(u, U)  \leq \xi(u, V)$, and hence
	\[\sup\{\xi(u, U) : u \in U \}  \leq \sup\{\xi(u, V) : u \in U \}, \]
	that is,
	$\mathcal{D}_{\xi}(U, U)  \leq \mathcal{D}_{\xi}(U, V), $
	and so
	\[\mathcal{D}_{\xi}(U, U)  \leq \max \{\mathcal{D}_{\xi}(U, V), \mathcal{D}_{\xi}(V, U)\}. \]
	Hence, $\mathcal{H}_{\xi}(U, U)  \leq \mathcal{H}_{\xi}(U, V)$.
\end{proof}

\section{Dislocated Metric Space Endowed with Directed Graph}
\label{secNaS1:DislocatedMetSpacewithDirectedGraph}
Existence of fixed points in ordered metric spaces has been studied by Ran and Reurings \cite{ran2004fixed}. Recently, many researchers have obtained fixed point results for single- and setvalued mappings defined on partially ordered metrics spaces (see, for example, \cite{beg2013fixed, harjani2009fixed, nicolae2011fixed, nieto2005contractive, nieto2007fixed, nieto2007existence, radenovic2010generalized}). Jachymski and Jozwik \cite{jachymski2007nonlinear} introduced a new approach in metric fixed point theory by replacing the order structure with a graph structure on a metric space.
Abbas et al. \cite{abbas2015fixed, abbas2016common, abbas2014fixed} obtained the fixed point of multivalued contaraction mappings on metric spaces with a directed graph. Then several useful fixed point results of single and multivalued mappings are appeared in \cite{alfuraidan2014remarks, acar2016multivalued, aleomraninejad2012some, bojor2010fixed, bojor2013jachymski, bojor2012fixed, chifu2012generalized, gwozdz2009ifs, jachymski2008contraction}. Recently, Latif et al. \cite{latif2019common} established some fixed point results for class of set-contraction mappings endowed with a directed graph.

In this work, we prove fixed
point results for set-valued maps, defined on the family of closed and bounded subsets of a
dislocated metric space endowed with a graph and satisfying generalzied graph $\phi$-contractive conditions. These
results extend and strengthen various known results in \cite{alber1997principle, assad1972fixed, banach1922operations, marin2011functions, nadler1969multi}.

Here, the dislocated metric space $(Y, \xi)$ is identified with a directed graph $G$ such that $G$ is weighted, that is, for each $(u, v) \in E(G)$ we assign the weight $\xi(u, v)$, where loops have weight $\xi(u, u) \geq 0$ for all $u \in V(G)$. Moreover, the $PH$-metric induced by the dislocated metric $\xi$ assigns a non-zero weight to every $V, W \in CB^{\xi}(Y)$, that is, $\mathcal{H}_{\xi}(V, W) \geq 0$, and also $\mathcal{H}_{\xi}(W, V) = 0$ implies $W=V$.

\begin{definition}
	\label{DefNaS1:1.3.1}
	Let $(Y, \xi)$ be a dislocated metric space and $T: CB^{\xi}(Y) \rightarrow CB^{\xi}(Y)$ be a set-valued mapping with set-valued domain endowed with a graph $G$. Then $T$ is said to be a generalized graph $\phi-$contraction, if the following conditions are satisfied:
	\begin{enumerate}[label=\textup{\arabic*)}, ref=\arabic*]
		\item An edge (path) between $U$ and $V$ implies and edge (path) between $T(U)$ and $T(V)$, for all $U, V \in CB^{\xi}(Y)$.
		\item There exists an upper semi-continuous and non-decreasing function $\phi : \mathbb{R_{+}} \rightarrow \mathbb{R_{+}}$ with $\phi(t) < t$ for each $t > 0$ such that for all $U, V \in CB^{\xi}(Y)$ with $(U, V) \subseteq E(G)$, it holds that
		\begin{align*}
			\mathcal{H}_{\xi}\Big(T(U), T(V)\Big) \leq \phi \Big(M_{T}(U, V)\Big),
		\end{align*}
	\end{enumerate}
where
\begin{multline*}
M_{T}(U, V) = \max\Big\{\mathcal{H}_{\xi}(U, V), \mathcal{H}_{\xi}(U, T(U)), \mathcal{H}_{\xi}(V, T(V)),  \\
\mathcal{H}_{\xi}(T(U), T(V)), \mathcal{H}_{\xi}(T^{2}(U), V), \mathcal{H}_{\xi}(T^{2}(U), T(V)),  \\
\frac{[\mathcal{H}_{\xi}(V, T(U)) + \mathcal{H}_{\xi}(U, T(V))]}{3} \Big\}.
	\end{multline*}
\end{definition}

\begin{definition}
	\label{DefNaS1:1.3.2}
	Let $(Y, \xi)$ be a dislocated metric space and $S: (Y, \xi) \rightarrow (Y, \xi)$ be a set-valued mapping with set-valued domain endowed with a graph $G$. Then $S$ is said to be a generalized rational graph $\phi-$contraction, if the following conditions are satisfied:
	\begin{enumerate}[label=\textup{\arabic*)}, ref=\arabic*]
		\item An edge (path) between $U$ and $V$ implies an edge (path) between $S(U)$ and $S(V)$, for all $U, V \in CB^{\xi}(Y)$. \label{dislocated_propS1}
		\item There exists an upper semi-continuous and non-decreasing function $\phi : \mathbb{R_{+}} \rightarrow \mathbb{R_{+}}$ with $\phi(t) < t$ for each $t > 0$ such that for all $U, V \in (Y, \xi)$ with $(U, V) \subseteq E(G)$, \label{dislocated_propS2}
		\begin{align*}
			\mathcal{H}_{\xi}\Big(S(U), S(V)\Big) \leq \phi \Big(N_{S}(U, V)\Big)
		\end{align*}
	\end{enumerate}
	is satisfied, where
	\begin{align*}
		N_{S}(U, V) &= \max\Big\{\frac{\mathcal{H}_{\xi}(U, S(V))[1 + \mathcal{H}_{\xi}(U, S(U))]}{2[1 + \mathcal{H}_{\xi}(U, V)]}, \notag \\
		&\qquad \qquad \frac{\mathcal{H}_{\xi}(V, S(V))[1 + \mathcal{H}_{\xi}(U, S(U))]}{1 + \mathcal{H}_{\xi}(U, V)}, \notag \\
		&\qquad \qquad \frac{\mathcal{H}_{\xi}(V, S(U))[1 + \mathcal{H}_{\xi}(U, S(U))]}{1 + \mathcal{H}_{\xi}(U, V)} \Big\}.
	\end{align*}
\end{definition}

\section{Fixed Point of Graphic Contractions in Dislocated Metric Space}
We present our main result in this section.

\begin{theorem}
	\label{ThmNaS1:1.4.1}
	Let $(Y, \xi)$ be a complete dislocated metric space equipped with a directed graph $G$, where $V = Y$ and $\Delta \subseteq CB^{\xi}(Y)$. If $T : CB^{\xi}(Y) \rightarrow CB^{\xi}(Y)$ is a generalized graphic $\phi-$contraction mapping, then the following statements hold:
	\begin{enumerate}[label=\textup{\arabic*)}, ref=\arabic*]
		\item If $F(T)=\{U\in CB^{\xi}(Y) \mid T(U)=U\}$ is complete, then the weight assigned to $U, V \in F(T)$ is zero. \label{dislocated_theorem-m1}
		\item $Y_{T} \neq \emptyset$ given that $F(T) \neq \emptyset$. \label{dislocated_theorem-m2}
		\item If $Y_{T} \neq \emptyset$ and $\widetilde{G}$ satisfies property $\mathcal{\overset{*}{P}}$, then $T$ has a fixed point. \label{dislocated_theorem-m3}
		\item $F(T)$ is complete if and only if $F(T)$ is singleton. \label{dislocated_theorem-m4}
	\end{enumerate}
\end{theorem}
\begin{proof}
	\ref{dislocated_theorem-m1}) \smartqed Let $U, V \in F(T)$, that is, $T(U) = U$ and $T(V) = V$ and suppose by way of contradiction that $\mathcal{H}_{\xi}(U, V) \neq 0$. Now, it follows that 
	\begin{align}
		\mathcal{H}_{\xi}(T(U), T(V)) &\leq \phi \big(M_{T}(U, V)\big), \label{teq1.1.3}
	\end{align}
	where
	\begin{align*}
		M_{T}(U, V) &= \max\Big\{\mathcal{H}_{\xi}(U, V), \mathcal{H}_{\xi}(U, T(U)), \mathcal{H}_{\xi}(V, T(V)), \notag \\
		&\qquad \qquad \mathcal{H}_{\xi}(T(U), T(V)), \mathcal{H}_{\xi}(T^{2}(U), V), \mathcal{H}_{\xi}(T^{2}(U), T(V)) \notag \\
		&\qquad \qquad \frac{[\mathcal{H}_{\xi}(V, T(U)) + \mathcal{H}_{\xi}(U, T(V))]}{3} \Big\} \\
		&= \mathcal{H}_{\xi}(U, V).
	\end{align*}
Then by (\ref{teq1.1.3}), we have
$\mathcal{H}_{\xi}(U, V) \leq \phi \big(\mathcal{H}_{\xi}(U, V)\big) < \mathcal{H}_{\xi}(U, V),$
	which is a contradiction. Hence our result follows, that is, $\mathcal{H}_{\xi}(U, V) = 0$.\\
	\noindent \ref{dislocated_theorem-m2}) Suppose that $U \in F(T)$ such that $T(U) = U$. Then we have $(U, T(U)) = (U, U) \in \Delta$, hence $U \in Y_{T}$.\\
	\noindent \ref{dislocated_theorem-m3}) Let $U_{0} \in Y_{T}$. It must be $\mathcal{H}_{\xi}(U_{0}, T(U_{0})) > 0$, since if it was otherwise, then we would have $U_{0}$ a fixed point of $T$. Now, since $Y = V(\tilde{G})$ there exists $U_{2} \in Y$ such that there is an edge between $T(U_{0}) = U_{1}$ and $U_{2}$, and $\mathcal{H}_{\xi}(U_{1}, T(U_{1})) > 0$, that is, $(U_{1}, U_{2}) \in E(\tilde{G})$, where $T(U_{1}) = U_{2}$.
Following, this argument we build a path $P : U_{0} - U_{1} - \cdots - U_{n}$ where $\{U_{n}\} \subseteq CB^{\xi}(Y)$ is a sequence. The previous argument also builds an iterative sequence by letting $U = U_{0}$ and $T(U_{i}) = U_{i+1}$ for $i = 1, 2, \ldots, n$. By definition of $T$,
	\begin{align*}
		\mathcal{H}_{\xi}(U), T^{n+1}(U)) &= \mathcal{H}_{\xi}(U_{n}, U_{n+1})= \mathcal{H}_{\xi}(T(U_{n-1}), T(U_{n})) \leq \phi \Big( M_{T}(U_{n-1}, U_{n}) \Big),
	\end{align*}
where
\begin{align*}
	&	M_{T}(U_{n-1}, U_{n}) = \\
& \max\Big\{\mathcal{H}_{\xi}(U_{n-1}, U_{n}), \mathcal{H}_{\xi}(U_{n-1}, T(U_{n-1})), \mathcal{H}_{\xi}(U_{n}, T(U_{n})), \\
		&\qquad \qquad \mathcal{H}_{\xi}(T(U_{n-1}), T(U_{n})), \mathcal{H}_{\xi}(T^{2}(U_{n-1}), U_{n}), \mathcal{H}_{\xi}(T^{2}(U_{n-1}), T(U_{n})), \notag \\
		&\qquad \qquad \frac{[\mathcal{H}_{\xi}(U_{n}, T(U_{n-1})) + \mathcal{H}_{\xi}(U_{n-1}, T(U_{n}))]}{3} \Big\} \\
		&= \max\Big\{\mathcal{H}_{\xi}(U_{n-1}, U_{n}), \mathcal{H}_{\xi}(U_{n}, U_{n+1}),\\
& \hspace{3cm} \frac{\mathcal{H}_{\xi}(U_{n}, U_{n}) + \mathcal{H}_{\xi}(U_{n-1}, U_{n+1})}{3}, \mathcal{H}_{\xi}(U_{n+1}, U_{n+1}) \Big\} \\
		&\leq \max\Big\{\mathcal{H}_{\xi}(U_{n-1}, U_{n}), \mathcal{H}_{\xi}(U_{n}, U_{n+1}),\\
& \hspace{3cm} \frac{\mathcal{H}_{\xi}(U_{n}, U_{n}) + \mathcal{H}_{\xi}(U_{n-1}, U_{n}) + \mathcal{H}_{\xi}(U_{n}, U_{n+1})}{3} \Big\}\\
		&= \max \Bigl\{\mathcal{H}_{\xi}(U_{n-1}, U_{n}), \mathcal{H}_{\xi}(U_{n}, U_{n+1})\Bigr\} \leq M_{T}(U_{n-1}, U_{n}).
	\end{align*}
	That is, $M_{T}(U_{n-1}, U_{n}) = \max \Bigl\{\mathcal{H}_{\xi}(U_{n-1}, U_{n}), \mathcal{H}_{\xi}(U_{n}, U_{n+1})\Bigr\} $.
If $M_{T}(U_{n-1}, U_{n}) = \mathcal{H}_{\xi}(U_{n}, U_{n+1})$, then we have a contradiction. Thus, we must have $M_{T}(U_{n-1}, U_{n}) = \mathcal{H}_{\xi}(U_{n - 1}, U_{n})$. It follows that
\begin{equation}
	\begin{array}{ll}
		\mathcal{H}_{\xi}(T^{n}(U), T^{n+1}(U)) &= \mathcal{H}_{\xi}(U_{n}, U_{n+1})  \\
		&= \mathcal{H}_{\xi}(T(U_{n-1}), T(U_{n}))  \\
		&\leq \phi \Big(\mathcal{H}_{\xi}(U_{n - 1}, U_{n})\Big)  \\
		&= \phi \Big(\mathcal{H}_{\xi}(T(U_{n-2}), T(U_{n-1})) \Big)  \\
		&\leq \phi^{2} \Big(\mathcal{H}_{\xi}(U_{n - 2}, U_{n-1})\Big)  \\
		&\qquad \vdots \qquad \vdots \qquad \vdots  \\
		&\leq \phi^{n} \Big(\mathcal{H}_{\xi}(U_{0}, T(U_{0}))\Big).
	\end{array}
\label{teq1.1.4}
\end{equation}
	\noindent Now for $m, n \in \mathbb{N}$ with $m > n$,
	\begin{align*}
		& \mathcal{H}_{\xi}(T^{n}(U), T^{m}(U)) \\  				
&\leq \mathcal{H}_{\xi}(T^{n}(U), T^{n+1}(U))+ \dots + \mathcal{H}_{\xi}(T^{n+k}(U), T^{n+k+1}(U)) \\
& \hspace{5cm} + \dots + \mathcal{H}_{\xi}(T^{m-1}(U), T^{m}(U)) \\
		&\leq \phi^{n} \Big(\mathcal{H}_{\xi}(U_{0}, T(U_{0}))\Big) + \dots + \phi^{n+k} \Big(\mathcal{H}_{\xi}(U_{0}, T(U_{0}))\Big) \\
& \hspace{5cm} + \cdots + \phi^{m-1} \Big(\mathcal{H}_{\xi}(U_{0}, T(U_{0}))\Big).
	\end{align*}
	\noindent It then follows that by taking the upper limits as $n \rightarrow \infty$ we have that the sequence $\{T^{n}(U)\}$ is a Cauchy. Since $(CB^{\xi}(Y), \mathcal{H}_{\xi})$ is a complete dislocated metric space, there exists $U^{*} \in CB^{\xi}(Y)$ such that $T^{n}(U) \rightarrow U^{*}$ as $n \rightarrow \infty$, that is, $\lim\limits_{n \rightarrow \infty} \mathcal{H}_{\xi}(T^{n}(U), U^{*}) = \lim\limits_{n \rightarrow \infty} \mathcal{H}_{\xi}(T^{n}(U), T^{n+1}(U)) = \mathcal{H}_{\xi}(U^{*}, U^{*})$.
It follows from inequality \eqref{teq1.1.4}, that $\mathcal{H}_{\xi}(T^{n}(U), T^{n+1}(U)) \leq \phi^{n}\big(\mathcal{H}_{\xi}(T(U^{*}), U^{*})\big)$, by taking the upper limit as $n \rightarrow \infty$ implies that $\mathcal{H}_{\xi}(T^{n}(U), T^{n+1}(U)) = 0$ and due to the uniqueness of limits we obtain  $\mathcal{H}_{\xi}(U^{*}, U^{*}) = 0$.
Then by property $\mathcal{\overset{*}{P}}$, there exists a subsequence $\{T^{n_{k}}(U)\}$ such that there's an edge between $T^{n_{k} - 1}(U)$ and $U^{*}$. Now by the triangle inequality \ref{dislocated_ax-d3}, we have
	\begin{align}
		\mathcal{H}_{\xi}(T(U^{*}), U^{*}) &\leq \mathcal{H}_{\xi}(T(U^{*}), T^{n_{k}}(U^{*})) + \mathcal{H}_{\xi}(T^{n_{k}}(U^{*}), U^{*}) \notag \\
		&\leq \phi \Big( M_{T}(U^{*}, T^{n_{k} - 1}(U^{*})) \Big) + \mathcal{H}_{\xi}(T^{n_{k}}(U^{*}), U^{*}),\label{teq1.1.5}
	\end{align}
	where
	\begin{align*}
		& M_{T}(T^{n_{k} - 1}(U^{*}), U^{*})\\
&= \max\Big\{\mathcal{H}_{\xi}(T^{n_{k} - 1}(U^{*}), U^{*}), \mathcal{H}_{\xi}(T^{n_{k} - 1}(U^{*}), T(T^{n_{k} - 1}(U^{*}))),  \\
		&\qquad \qquad \mathcal{H}_{\xi}(U^{*}, T(U^{*})), \mathcal{H}_{\xi}(T(T^{n_{k} - 1}(U^{*})), T(U^{*})),\\
		&\qquad \qquad \mathcal{H}_{\xi}(T^{2}(T^{n_{k} - 1}(U^{*})), U^{*}), \mathcal{H}_{\xi}(T^{2}(T^{n_{k} - 1}(U^{*})), T(U^{*})) \notag \\
		&\qquad \qquad \frac{(\mathcal{H}_{\xi}(U^{*}, T(T^{n_{k} - 1}(U^{*}))) + \mathcal{H}_{\xi}(T^{n_{k} - 1}(U^{*}), T(U^{*})))}{3} \Big\}.
	\end{align*}
	
	\noindent \textbf{Case} $1$: If $M_{T}(T^{n_{k} - 1}(U^{*}), U^{*}) =  \mathcal{H}_{\xi}(T^{n_{k} - 1}(U^{*}), U^{*})$, then by inequality \eqref{teq1.1.5} we have
	\begin{align*}
		\mathcal{H}_{\xi}(T(U^{*}), U^{*}) & \leq \phi \Big(H_{\rho}(T^{n_{k} - 1}(U^{*}), U^{*})\Big) + H_{\rho}(T^{n_{k}}(U^{*}), U^{*})
	\end{align*}
and taking the upper limit $n \rightarrow \infty$, we have
	\begin{align*}
		\mathcal{H}_{\xi}(T(U^{*}), U^{*}) & \leq \phi (0) + 0
	\end{align*}
	and hence $T(U^{*}) = U^{*}$.
	
	\noindent \textbf{Case} $2$: If $M_{T}(T^{n_{k} - 1}(U^{*}), U^{*}) = \mathcal{H}_{\xi}(T^{n_{k} - 1}(U^{*}), T^{n_{k}}(U^{*}))$, then by inequality \eqref{teq1.1.5}, we have	
	\begin{align*}
		\mathcal{H}_{\xi}(T(U^{*}), U^{*}) & \leq \phi \Big(\mathcal{H}_{\xi}(T^{n_{k} - 1}(U^{*}), T^{n_{k}}(U^{*}))\Big) + \mathcal{H}_{\xi}(T^{n_{k}}(U^{*}), U^{*})
	\end{align*}
	and by taking the upper limit $n \rightarrow \infty$, we have
	\begin{align*}
		\mathcal{H}_{\xi}(T(U^{*}), U^{*}) & \leq \phi (0) + 0
	\end{align*}
and hence $T(U^{*}) = U^{*}$.
	
	\noindent \textbf{Case} $3$: If $M_{T}(T^{n_{k} - 1}(U^{*}), U^{*}) = \mathcal{H}_{\xi}(U^{*}, T(U^{*}))$, then by inequality \eqref{teq1.1.5}, we have
	\begin{align*}
		\mathcal{H}_{\xi}(T(U^{*}), U^{*}) & \leq \phi \Big(\mathcal{H}_{\xi}(U^{*}, T(U^{*}))\Big) + \mathcal{H}_{\xi}(T^{n_{k}}(U^{*}), U^{*})
	\end{align*}
and on taking the upper limit $n \rightarrow \infty$, we have:
\begin{align*}
		\mathcal{H}_{\xi}(T(U^{*}), U^{*}) & \leq \phi \Big(\mathcal{H}_{\xi}(U^{*}, T(U^{*}))\Big) + 0
		< \mathcal{H}_{\xi}(U^{*}, T(U^{*}))
\end{align*}
and hence we must have: $\mathcal{H}_{\xi}(T(U^{*}), U^{*}) = 0$. \\
	
\noindent \textbf{Case} $4$: If $M_{T}(T^{n_{k} - 1}(U^{*}), U^{*}) = \mathcal{H}_{\xi}(T^{n_{k} + 1}(U^{*}), U^{*})$, then by inequality \eqref{teq1.1.5} we have
	\begin{align*}
		\mathcal{H}_{\xi}(T(U^{*}), U^{*}) & \leq \phi \Big(\mathcal{H}_{\xi}(T^{n_{k} + 1}(U^{*}), U^{*})\Big) + \mathcal{H}_{\xi}(T^{n_{k}}(U^{*}), U^{*})
	\end{align*}
and on taking the upper limit $n \rightarrow \infty$, we have
	\begin{align*}
		\mathcal{H}_{\xi}(T(U^{*}), U^{*}) & \leq \phi (0) + 0,
	\end{align*}
and hence we must have $\mathcal{H}_{\xi}(T(U^{*}), U^{*}) = 0$. \\
	
	\noindent \textbf{Case} $5$: If $M_{T}(T^{n_{k} - 1}(U^{*}), U^{*})
= \mathcal{H}_{\xi}(T^{n_{k} + 1}(U^{*}), U^{*})$, then by inequality \eqref{teq1.1.5} we have
	\begin{align*}
		\mathcal{H}_{\xi}(T(U^{*}), U^{*}) & \leq \phi \Big(\mathcal{H}_{\xi}(T^{n_{k} + 1}(U^{*}), U^{*})\Big) + \mathcal{H}_{\xi}(T^{n_{k}}(U^{*}), U^{*})
	\end{align*}
	and on taking the upper limit $n \rightarrow \infty$, we have
	\begin{align*}
		\mathcal{H}_{\xi}(T(U^{*}), U^{*}) & \leq \phi (0) + 0
	\end{align*}
	and hence we must have  $\mathcal{H}_{\xi}(T(U^{*}), U^{*}) = 0$.  \\
	
	\noindent \textbf{Case} $6$: If $M_{T}(T^{n_{k} - 1}(U^{*}), U^{*}) = \mathcal{H}_{\xi}(T^{n_{k} + 1}(U^{*}), T(U^{*}))$, then by inequality \eqref{teq1.1.5} we have
	\begin{align*}
		\mathcal{H}_{\xi}(T(U^{*}), U^{*}) & \leq \phi \Big(\mathcal{H}_{\xi}(T^{n_{k} + 1}(U^{*}), 			T(U^{*}))\Big) + \mathcal{H}_{\xi}(T^{n_{k}}(U^{*}), U^{*}) \\
		&\leq \phi\Big(\mathcal{H}_{\xi}(T^{n_{k} + 1}(U^{*}), U^{*}) + \mathcal{H}_{\xi}(U^{*}, T(U^{*}))\Big) \\
		&+ \mathcal{H}_{\xi}(T^{n_{k}}(U^{*}), U^{*})
	\end{align*}
taking the upper limit $n \rightarrow \infty$, we have
	\begin{align*}
		\mathcal{H}_{\xi}(T(U^{*}), U^{*}) & \leq \phi \Big(\mathcal{H}_{\xi}(U^{*}, T(U^{*})) + 0 \Big) + 0
	\end{align*}
and hence as in \textbf{Case} $5$, we must have: $\mathcal{H}_{\xi}(T(U^{*}), U^{*}) = 0$. \\
	
	\noindent \textbf{Case} $7$: If $M_{T}(T^{n_{k} - 1}(U^{*}), U^{*}) = \frac{[\mathcal{H}_{\xi}(U^{*}, T^{n_{k}}(U^{*})) + \mathcal{H}_{\xi}(T^{n_{k} - 1}(U^{*}), T(U^{*}))]}{3}$, then by inequality \eqref{teq1.1.5} we have
	\begin{align*}
		&\mathcal{H}_{\xi}(T(U^{*}), U^{*}) \leq \phi \Big(\frac{\mathcal{H}_{\xi}(U^{*}, T^{n_{k}}(U^{*})) + \mathcal{H}_{\xi}(T^{n_{k} - 1}(U^{*}), T(U^{*}))}{2}\Big) \\
& \quad + \mathcal{H}_{\xi}(T^{n_{k}}(U^{*}), U^{*}) \\
		& \leq \phi \Big(\frac{\mathcal{H}_{\xi}(U^{*}, T^{n_{k}}(U^{*})) + \mathcal{H}_{\xi}(T^{n_{k} - 1}(U^{*}), U^{*}) + \mathcal{H}_{\xi}(U^{*}, T(U^{*}))}{3}\Big)  \\
		& \quad + \mathcal{H}_{\xi}(T^{n_{k}}(U^{*}), U^{*}).
	\end{align*}
Then by taking the upper limit $n \rightarrow \infty$, and as in \textbf{Case} $6$ we get $\mathcal{H}_{\xi}(T(U^{*}), U^{*}) = 0$.
	
\noindent \ref{dislocated_theorem-m4}): Suppose by way of contradiction that $|F(T)| > 1$, that is, for
$U, V \in F(T)$ such that $\mathcal{H}_{\xi}(T(U), T(V)) = \mathcal{H}_{\xi}(U, V) > 0$,
	\begin{align*}
		0 < \mathcal{H}_{\xi}(U, V) &= \mathcal{H}_{\xi}(T(U), T(V))
		\leq \phi \Big(M_{T}(U, V)\Big),
	\end{align*}
where after applying the definition of $M_{T}$, we obtain $M_{T}(U, V) = \mathcal{H}_{\xi}(U, V)$. Thus
$0 < \mathcal{H}_{\xi}(U, V) \leq \phi \Big(\mathcal{H}_{\xi}(U, V)\Big)$, and a contradiction follows. Conversely, it follows that if $U \in F(T)$, then $T(U) = U$, clearly we have $(U, T(U)) \subseteq E(G)$ and since $F(T)$ is a singleton, it is complete.
\qed \end{proof}

\begin{corollary}
	\label{CorNaS1:1.4.1}
	Let $(CB^{\xi}(Y), \mathcal{H}_{\xi})$ be a complete dislocated metric space.

If $T: CB^{\xi}(Y) \rightarrow CB^{\xi}(Y)$ is a multi-valued mapping with set-valued domain such that for all $U, V \in CB^{\xi}(Y)$, we have
	\begin{align*}
		\mathcal{H}_{\xi}(T(U), T(V)) &\leq \lambda M_{T}(U,V),
	\end{align*}
	is satisfied with $\lambda \in [0, 1) $, where
	\begin{align*}
		M_{T}(U, V) &= \max\Big\{\mathcal{H}_{\xi}(U, V), \mathcal{H}_{\xi}(U, T(U)), \mathcal{H}_{\xi}(V, T(V)), \notag \\
		&\qquad \qquad \mathcal{H}_{\xi}(T(U), T(V)), \mathcal{H}_{\xi}(T^{2}(U), V), \mathcal{H}_{\xi}(T^{2}(U), T(V)), \notag \\
		&\qquad \qquad \frac{[\mathcal{H}_{\xi}(V, T(U)) + \mathcal{H}_{\xi}(U, T(V))]}{3} \Big\}.
	\end{align*}
	Then $T$ has at most one fixed point.
\end{corollary}

\begin{proof} \smartqed
	Let $\phi(t) = \lambda t$ for $\lambda \in [0, 1)$, then by Theorem \ref{ThmNaS1:1.4.1} we get our result.
\qed \end{proof}	
\begin{theorem}
	\label{ThmNaS1:1.4.2}
	Let $(Y, \xi)$ be a complete dislocated metric space equipped with a directed graph $G$, where $V = Y$ and $\Delta \subseteq CB^{\xi}(Y)$. If $S : CB^{\xi}(Y) \rightarrow CB^{\xi}(Y)$ is a generalized rational graphic $\phi-$contraction mapping. Then the following statements hold:
\begin{enumerate}[label=\textup{\arabic*)}, ref=\arabic*]
		\item If $F(S)$ is complete, then the weight assigned to $U, V \in F(S)$ is zero. \label{dislocated_theorem-n1}
		\item $Y_{S} \neq \emptyset$ given that $F(S) \neq \emptyset$. \label{dislocated_theorem-n2}
		\item If $Y_{S} \neq \emptyset$ and $\tilde{G}$ satisfies property $\mathcal{\overset{*}{P}}$, then $S$ has a fixed point. \label{dislocated_theorem-n3}
		\item $F(S)$ is complete if and only if $F(S)$ is singleton. \label{dislocated_theorem-n4}
	\end{enumerate}
\end{theorem}

\begin{proof} \smartqed 
\ref{dislocated_theorem-n1}) Suppose otherwise that $U, V \in F(S)$ and that $\mathcal{H}_{\xi}(U, V) > 0$. Then since $S(U) = U$ and $S(V) = V$, it follows that
	\begin{align*}
		\mathcal{H}_{\xi}(S(U), S(V)) &\leq \phi \Big(N_{S}(U, V)\Big),
	\end{align*}
	where
	\begin{align*}
		N_{S}(U, V) &= \max\Big\{\frac{\mathcal{H}_{\xi}(U, S(V))[1 + \mathcal{H}_{\xi}(U, S(U))]}{2[1 + \mathcal{H}_{\xi}(U, V)]}, \notag \\
		&\qquad \qquad \frac{\mathcal{H}_{\xi}(V, S(V))[1 + \mathcal{H}_{\xi}(U, S(U))]}{1 + \mathcal{H}_{\xi}(U, V)}, \notag \\
		&\qquad \qquad \frac{\mathcal{H}_{\xi}(V, S(U))[1 + \mathcal{H}_{\xi}(U, S(U))]}{1 + \mathcal{H}_{\xi}(U, V)} \Big\},
	\end{align*}
and	
	\begin{align*}
		N_{S}(U, V) &= \max\Big\{\frac{\mathcal{H}_{\xi}(U, V)}{2[1 + \mathcal{H}_{\xi}(U, V)]}, \frac{\mathcal{H}_{\xi}(V, U)}{1 + \mathcal{H}_{\xi}(U, V)} \Big\} \\
		& = \frac{\mathcal{H}_{\xi}(V, U)}{1 + \mathcal{H}_{\xi}(U, V)}.
	\end{align*}
Then we have $0 < \mathcal{H}_{\xi}(S(U), S(V)) = \mathcal{H}_{\xi}(U, V) \leq \phi \Big(\frac{\mathcal{H}_{\xi}(V, U)}{1 + \mathcal{H}_{\xi}(U, V)}\Big)$.
Therefore, $$(1 + \mathcal{H}_{\xi}(U, V))\mathcal{H}_{\xi}(U, V) < \mathcal{H}_{\xi}(U, V),$$ which is a contradiction, and our result follows.\\
	
	\noindent \ref{dislocated_theorem-n2}) Let $U \in F(S)$, then since $S(U) = U$, it follows that $(U, S(U)) \subseteq E(G)$, and hence $U \in Y_{S}$. \\
	
	\noindent \ref{dislocated_theorem-n3}) Let $U_{0} \in Y_{S}$. It must be $\mathcal{H}_{\xi}(U_{0}, S(U_{0})) > 0$, since if it was otherwise, then we would have $U_{0}$ a fixed point for $S$. Now, by definition of $S$ there exists $U_{2} \in CB^{\xi}(Y)$ such that there is an edge between $S(U_{0}) = U_{1}$ and $U_{2}$, and $\mathcal{H}_{\xi}(U_{1}, S(U_{1})) > 0$, that is, $(U_{1}, U_{2}) \subseteq E(\tilde{G})$, where $S(U_{1}) = U_{2}$. Following, this argument we build a path $P : U_{0} - U_{1} - \cdots - U_{n}$, where $\{U_{n}\} \subseteq CB^{\xi}(Y)$ is a sequence. The previous argument also builds an iterative sequence by letting $U = U_{0}$ and $S(U_{i}) = U_{i+1}$ for $i = 1, 2, \ldots, n$. By definition of $S$,
	\begin{align*}
		\mathcal{H}_{\xi}(S^{n}(U), S^{n+1}(U)) &= \mathcal{H}_{\xi}(U_{n}, U_{n+1})= \mathcal{H}_{\xi}(S(U_{n-1}), S(U_{n})) \\
		&\leq \phi \Big( N_{S}(U_{n-1}, U_{n}) \Big),
	\end{align*}
	where
	\begin{align*}
		N_{S}(U_{n-1}, U_{n}) &= \max\Big\{\frac{\mathcal{H}_{\xi}(U_{n-1}, S(U_{n}))[1 + \mathcal{H}_{\xi}(U_{n-1}, S(U_{n-1}))]}{2[1 + \mathcal{H}_{\xi}(U_{n-1}, U_{n})]}, \notag \\
		&\qquad \qquad \frac{\mathcal{H}_{\xi}(U_{n}, S(U_{n}))[1 + \mathcal{H}_{\xi}(U_{n-1}, S(U_{n-1}))]}{1 + \mathcal{H}_{\xi}(U_{n-1}, U_{n})}, \\
		&\qquad \qquad \frac{\mathcal{H}_{\xi}(U_{n}, S(U_{n-1}))[1 + \mathcal{H}_{\xi}(U_{n-1}, S(U_{n-1}))]}{1 + \mathcal{H}_{\xi}(U_{n-1}, U_{n})} \Big\}.
	\end{align*}
	Then
	\begin{align*}
		N_{S}(U_{n-1}, U_{n}) &= \max\Big\{\frac{\mathcal{H}_{\xi}(U_{n-1}, U_{n+1})}{2}, \mathcal{H}_{\xi}(U_{n}, U_{n+1}), \mathcal{H}_{\xi}(U_{n}, U_{n}) \Big\} \\
		&\leq \max\Big\{\frac{\mathcal{H}_{\xi}(U_{n-1}, U_{n}) + \mathcal{H}_{\xi}(U_{n}, U_{n+1})}{2}, \mathcal{H}_{\xi}(U_{n}, U_{n+1})\Big\} \\
		&\leq \max \Big\{\mathcal{H}_{\xi}(U_{n-1}, U_{n}), \mathcal{H}_{\xi}(U_{n}, U_{n+1})\Big\}.
	\end{align*}
Now, it follows that if $\Big\{\mathcal{H}_{\xi}(U_{n-1}, U_{n}), \mathcal{H}_{\xi}(U_{n}, U_{n+1})\Big\} = \mathcal{H}_{\xi}(U_{n}, U_{n+1})$, then we have a contradiction. Hence, we must have  $\Big\{\mathcal{H}_{\xi}(U_{n-1}, U_{n}), \mathcal{H}_{\xi}(U_{n}, U_{n+1})\Big\} = \mathcal{H}_{\xi}(U_{n-1}, U_{n})$. It then follows that
	\begin{align}
		\mathcal{H}_{\xi}(S^{n}(U), S^{n+1}(U)) &= \mathcal{H}_{\xi}(U_{n}, U_{n+1}) = \mathcal{H}_{\xi}(S(U_{n-1}), S(U_{n})) \notag \\
		&\leq \phi \Big(\mathcal{H}_{\xi}(U_{n - 1}, U_{n})\Big) \notag \\
		&= \phi \Big(\mathcal{H}_{\xi}(S(U_{n-2}), S(U_{n-1})) \Big) \notag \\
		&\leq \phi^{2} \Big(\mathcal{H}_{\xi}(U_{n - 2}, U_{n-1})\Big) \notag \\
		&\qquad \vdots \qquad \vdots \qquad \vdots \notag \\
		&\leq \phi^{n} \Big(\mathcal{H}_{\xi}(U_{0}, S(U_{0}))\Big). \label{teq1.1.6}
	\end{align}
Now for $m, n \in \mathbb{N}$ with $m > n$,
	\begin{align*}
		&\mathcal{H}_{\xi}(S^{n}(U), S^{m}(U))\leq \mathcal{H}_{\xi}(S^{n}(U), S^{n+1}(U))\\  				 &+ \mathcal{H}_{\xi}(S^{n+1}(U), S^{n+2}(U)) + \cdots + \mathcal{H}_{\xi}(S^{m-1}(U), S^{m}(U)) \\
		&\leq \phi^{n} \Big(\mathcal{H}_{\xi}(U_{0}, S(U_{0}))\Big) + \phi^{n+1} \Big(\mathcal{H}_{\xi}(U_{0}, S(U_{0}))\Big) + \cdots + \phi^{m-1} \Big(\mathcal{H}_{\xi}(U_{0}, S(U_{0}))\Big).
	\end{align*}
It then follows that by taking the upper limits as $n \rightarrow \infty$ we have that the sequence $\{S^{n}(U)\}$ is a Cauchy. Since $(CB^{\xi}(Y), \mathcal{H}_{\xi})$ is a complete dislocated metric space, there exists $U^{*} \in CB^{\xi}(Y)$ such that $S^{n}(U) \rightarrow U^{*}$ as $n \rightarrow \infty$; that is, $\lim\limits_{n \rightarrow \infty} \mathcal{H}_{\xi}(S^{n}(U), U^{*}) = \lim\limits_{n \rightarrow \infty} \mathcal{H}_{\xi}(S^{n}(U), S^{n+1}(U)) = \mathcal{H}_{\xi}(U^{*}, U^{*})$.
It follows from inequality \eqref{teq1.1.6}, that $\mathcal{H}_{\xi}(S^{n}(U), S^{n+1}(U)) \leq \phi^{n}\big(\mathcal{H}_{\xi}(S(U^{*}), U^{*})\big)$, by taking the upper limit as $n \rightarrow \infty$ implies that $\mathcal{H}_{\xi}(S^{n}(U), S^{n+1}(U)) = 0$ and due to the uniqueness of limits we obtain: $\mathcal{H}_{\xi}(U^{*}, U^{*}) = 0$.
Now we are to show that $U^{*}=S(U^{*})$. Assume otherwise that $\mathcal{H}_{\xi}(S(U^{*}), U^{*})>0$. Then by property $\mathcal{\overset{*}{P}}$, there exists a subsequence $\{S^{n_{k}}(U)\}$ such that there's an edge between $S^{n_{k} - 1}(U)$ and $U^{*}$. Now by the triangle inequality \ref{dislocated_ax-d3}, we have
	\begin{align*}
		H_{\rho}(S(U^{*}), U^{*})
		&\leq H_{\rho}(S^{n_{k}}(U^{*}),S(U^{*})) + H_{\rho}(S^{n_{k}}(U^{*}), U^{*}) \notag \\
		&\leq \phi \Big( N_{S}(S^{n_{k} - 1}(U^{*}),U^{*}) \Big) + H_{\rho}(S^{n_{k}}(U^{*}), U^{*}),
	\end{align*}
	where
	\begin{eqnarray*}
		N_{S}(S^{n_{k}-1}U^{\ast },U^{\ast}) &=&\max \{ \frac{\mathcal{H}_{\xi}(S^{n_{k}-1}U^{\ast },S(U^{\ast }))[1+\mathcal{H}_{\xi}(S^{n_{k}-1}U^{\ast
			},S(S^{n_{k}-1}U^{\ast }))]}{2[1+\mathcal{H}_{\xi}(S^{n_{k}-1}U^{\ast },U^{\ast })]},
		\\
		&&\qquad  \frac{\mathcal{H}_{\xi}(U^{\ast },S(U^{\ast }))[1+\mathcal{H}_{\xi}(S^{n_{k}-1}U^{\ast
			},S(S^{n_{k}-1}U^{\ast }))]}{1+\mathcal{H}_{\xi}(S^{n_{k}-1}U^{\ast },U^{\ast })}, \\
		&&\qquad  \frac{\mathcal{H}_{\xi}(U^{\ast },S(S^{n_{k}-1}U^{\ast }))[1+\mathcal{H}_{\xi}(S^{n_{k}-1}U^{\ast },S(S^{n_{k}-1}U^{\ast }))]}{1+\mathcal{H}_{\xi}(S^{n_{k}-1}U^{\ast },U^{\ast })}\}\\
		&&=\max \{ \frac{\mathcal{H}_{\xi}(S^{n_{k}-1}U^{\ast },S(U^{\ast }))[1+\mathcal{H}_{\xi}(S^{n_{k}-1}U^{\ast
			},S^{n_{k}}U^{\ast })]}{2[1+\mathcal{H}_{\xi}(S^{n_{k}-1}U^{\ast },U^{\ast })]}, \\
		&&\qquad \qquad \frac{\mathcal{H}_{\xi}(U^{\ast },S(U^{\ast }))[1+\mathcal{H}_{\xi}(S^{n_{k}-1}U^{\ast
			},S^{n_{k}-1}U^{\ast })]}{1+\mathcal{H}_{\xi}(S^{n_{k}-1}U^{\ast },U^{\ast })}, \\
		&&\qquad \qquad \frac{\mathcal{H}_{\xi}(U^{\ast },S^{n_{k}}U^{\ast })[1+\mathcal{H}_{\xi}(S^{n_{k}-1}U^{\ast },S^{n_{k}}U^{\ast })]}{1+\mathcal{H}_{\xi}(S^{n_{k}-1}U^{\ast
			},U^{\ast })}\}.
	\end{eqnarray*}
We will consider cases taking each of $N_{S}(S^{n_{k} - 1}(U^{*}),U^{*})$.\\
	
	\noindent \textbf{Case $1$:}
	\begin{align*}
		& \mathcal{H}_{\xi}(S(U^{*}), U^{*}) \leq \mathcal{H}_{\xi}( S^{n_{k}}(U^{*}),S(U^{*})) + \mathcal{H}_{\xi}(S^{n_{k}}(U^{*}), U^{*})  \\
		&\leq \phi \Big( \frac{\mathcal{H}_{\xi}(S^{n_{k}-1}U^{\ast },S(U^{\ast }))[1+\mathcal{H}_{\xi}(S^{n_{k}-1}U^{\ast
			},S^{n_{k}}U^{\ast })]}{2[1+H\mathcal{H}_{\xi}(S^{n_{k}-1}U^{\ast },U^{\ast })]} \Big) + \mathcal{H}_{\xi}(S^{n_{k}}(U^{*}), U^{*})\\
		&\leq \phi \Big( \frac{[\mathcal{H}_{\xi}(S^{n_{k}-1}U^{\ast },U^{\ast })+\mathcal{H}_{\xi}(U^{\ast },S(U^{\ast }))][1+\mathcal{H}_{\xi}(S^{n_{k}-1}U^{\ast
			},S^{n_{k}}U^{\ast })]}{2[1+\mathcal{H}_{\xi}(S^{n_{k}-1}U^{\ast },U^{\ast })]} \Big) \\
		& \ \ \ \ + \mathcal{H}_{\xi}(S^{n_{k}}(U^{*}), U^{*})
	\end{align*}
and by taking the upper limit $n_{k} \rightarrow \infty$ we get
\begin{align*}
		\mathcal{H}_{\xi}(S(U^{*}), U^{*}) &\leq \phi(\mathcal{H}_{\xi}(U^{\ast },S(U^{\ast })) + 0
		<\mathcal{H}_{\xi}(S(U^{*}), U^{*}),
	\end{align*}
which is a contradiction. Hence, $S(U^{*})= U^{*}$. \\
	
	\noindent \textbf{Case $2$:}
	\begin{align*}
		& \mathcal{H}_{\xi}(S(U^{*}), U^{*})\leq \mathcal{H}_{\xi}( S^{n_{k}}(U^{*}),S(U^{*})) + \mathcal{H}_{\xi}(S^{n_{k}}(U^{*}), U^{*})  \\
		&\leq \phi \Big( \frac{\mathcal{H}_{\xi}(U^{\ast },S(U^{\ast }))[1+\mathcal{H}_{\xi}(S^{n_{k}-1}(U^{\ast }),S^{n_{k}-1}(U^{\ast }))]}{1+\mathcal{H}_{\xi}(S^{n_{k}-1}U^{\ast },U^{\ast })} \Big) + \mathcal{H}_{\xi}(S^{n_{k}}(U^{*}), U^{*})
	\end{align*}
and by taking the upper limit $n_{k} \rightarrow \infty$ we get
\begin{align*}
		\mathcal{H}_{\xi}(S(U^{*}), U^{*}) &\leq \phi(\mathcal{H}_{\xi}(U^{\ast },S(U^{\ast })) + 0
		<\mathcal{H}_{\xi}(S(U^{*}), U^{*}),
	\end{align*}
which is a contradiction. Hence, $S(U^{*})= U^{*}$. \\
	
	\noindent \textbf{Case $3$:}
	\begin{align*}
		& \mathcal{H}_{\xi}(S(U^{*}), U^{*})\leq \mathcal{H}_{\xi}( S^{n_{k}}(U^{*}),S(U^{*})) + \mathcal{H}_{\xi}(S^{n_{k}}(U^{*}), U^{*})  \\
		&\leq \phi \Big( \frac{\mathcal{H}_{\xi}(U^{\ast },S^{n_{k}}(U^{\ast }))[1+\mathcal{H}_{\xi}(S^{n_{k}-1}(U^{\ast }),S^{n_{k}}(U^{\ast }))]}{1+\mathcal{H}_{\xi}(S^{n_{k}-1}U^{\ast
			},U^{\ast })} \Big) + \mathcal{H}_{\xi}(S^{n_{k}}(U^{*}), U^{*})
	\end{align*}
and by taking the upper limit $n_{k} \rightarrow \infty$, we get
	\begin{align*}
		\mathcal{H}_{\xi}(S(U^{*}), U^{*}) &\leq \phi(0) + 0,
	\end{align*}
	a contradiction and hence $S(U^{*})= U^{*}$. \\
	Thus from all cases we obtain that  $S(U^{*})= U^{*}$. \\

	\noindent \ref{dislocated_theorem-m4}) Suppose by way of contradiction that $|F(S)| > 1$, that is $U, V \in F(S)$ such that $\mathcal{H}_{\xi}(T(U), T(V)) = \mathcal{H}_{\xi}(U, V) > 0$. Then
\begin{align*}
		0 < \mathcal{H}_{\xi}(U, V) &= \mathcal{H}_{\xi}(S(U), S(V)) \leq \phi \Big(N_{S}(U, V)\Big),
\end{align*}
where
	\begin{eqnarray*}
		N_{S}(U,V) &=&\max \Big\{ \frac{\mathcal{H}_{\xi}(U,S(V))[1+\mathcal{H}_{\xi}(U,S(U))]}{%
			2[1+\mathcal{H}_{\xi}(U,V)]}, \\
		&& \qquad \qquad \frac{\mathcal{H}_{\xi}(V,S(V))[1+\mathcal{H}_{\xi}(U,S(U))]}{1+\mathcal{H}_{\xi}(U,V)},\\
		&& \qquad \qquad \frac{\mathcal{H}_{\xi}(V,S(U))[1+\mathcal{H}_{\xi}(U,S(U))]}{1+\mathcal{H}_{\xi}(U,V)}\Big\}\\
		&=&\mathcal{H}_{\xi}(U,V),
	\end{eqnarray*}
implying $0 < \mathcal{H}_{\xi}(U, V) \leq \phi \Big(\mathcal{H}_{\xi}(U, V)\Big)$, whence a contradiction follows.
Conversely it follows that if $U \in F(S)$, then $S(U) = U$, clearly we have $(U, S(U)) \subseteq E(G)$ and since $F(S)$ is a singleton.
\qed \end{proof}

\begin{corollary}
	\label{CorNaS1:1.4.2}
	Let $(CB^{\xi}(Y), \mathcal{H}_{\xi})$ be a complete dislocated metric space.

If $T: CB^{\xi}(Y) \rightarrow CB^{\xi}(Y)$ is a multi-valued mapping with set-valued domain such that for all $U, V \in CB^{\xi}(Y)$, we have
\begin{align*}
\mathcal{H}_{\xi}(T(U), T(V)) &\leq \lambda \mathcal{H}_{\xi}(U, V) \text{ where } \lambda \in [0, 1),
	\end{align*}
then $T$ has a fixed point.
\end{corollary}

\begin{proof} \smartqed
	Let $\phi(t) = \lambda t$ for $\lambda \in [0, 1)$, then by Theorem \ref{ThmNaS1:1.4.1} and Theorem \ref{ThmNaS1:1.4.2}, we obtain our desired result.
\qed \end{proof}

\noindent We also state a more general result in the following Corollary.

\begin{corollary}
	\label{CorNaS1:1.4.3}
	Let $(CB^{\xi}(Y), \mathcal{H}_{\xi})$ be a complete dislocated metric space.

If $S: CB^{\xi}(Y) \rightarrow CB^{\xi}(Y)$ is a multi-valued mapping such that for all $U, V \in CB^{\xi}(Y)$, we have
	\begin{align*}
		\mathcal{H}_{\xi}(S(U), S(V)) &\leq \lambda N_{S}(U, V) \text{ for } \lambda \in [0, 1),
	\end{align*}
	where			
	\begin{align*}
		N_{S}(U, V) &= \max\Big\{\frac{\mathcal{H}_{\xi}(U, S(V))[1 + \mathcal{H}_{\xi}(U, S(U))]}{2[1 + \mathcal{H}_{\xi}(U, V)]}, \notag \\
		&\qquad \qquad \frac{\mathcal{H}_{\xi}(V, S(V))[1 + \mathcal{H}_{\xi}(U, S(U))]}{1 + \mathcal{H}_{\xi}(U, V)}, \notag \\
		&\qquad \qquad \frac{\mathcal{H}_{\xi}(V, S(U))[1 + \mathcal{H}_{\xi}(U, S(U))]}{1 + \mathcal{H}_{\xi}(U, V)} \Big\},
	\end{align*}
	then $T$ has a fixed point.
\end{corollary}

\begin{proof} \smartqed
	Let $\phi(t) = \lambda t$ for $\lambda \in [0, 1)$, then by Theorem \ref{ThmNaS1:1.4.2} we get our result.
\qed \end{proof}



\end{document}